\documentclass{amsart}
\usepackage[utf8]{inputenc}
\usepackage[left=1.25in,right=1.25in,top=0.75in,bottom=0.75in]{geometry}

\usepackage{amssymb,amsmath,amsthm, mathtools, mathrsfs, stmaryrd}
\usepackage[dvipsnames]{xcolor}
\usepackage{tikz-cd}
\usetikzlibrary{decorations.pathmorphing}
\usepackage{url}
\usepackage{graphicx}
\usepackage{hyperref}
\hypersetup{
  colorlinks,
  citecolor= Cerulean,
  linkbordercolor=pink, 
  linkcolor=red,
  urlcolor= green
 }
\tikzset{
    dot/.style={
        node contents={},
        circle,
        fill,
        inner sep=1pt,
    },
    big dot/.style={
        node contents={},
        circle,
        fill=gray,
        inner sep=2.5pt,
    },
    fuzzy/.style={
        node contents={},
        circle,
        inner color=black,
        outer color=white,
        inner sep=6pt,
    },
    every label/.style={
        inner sep=1pt,
    },
    every pin/.style={
        pin edge={thin, black, shorten <=3pt},
        pin distance=1cm,
        anchor=north,
        inner sep=1pt,
    }
}
\hypersetup{linktocpage}
\theoremstyle{definition}
\newtheorem{defi}{Definition}[section]
\newtheorem{theorem}{Theorem}[section]
\newtheorem{lemma}{Lemma}[section]

\newtheorem{prop}{Proposition}[section]
\newtheorem{coro}{Corollary}[section]

\DeclareMathOperator{\spec}{Spec}
\DeclareMathOperator{\im}{im}

\DeclareMathOperator{\proj}{Proj}

\newtheorem{ex}{Example}[section]
\newcommand{\aff}{\mathbf A}
\newcommand{\ff}{\mathbf F}
\newcommand{\rr}{\mathbf R}

\newcommand{\zz}{\mathbf Z}

\newcommand{\qq}{\mathbf Q}

\newcommand{\pp}{\prime}
\newcommand{\Mod}[1]{\ (\mathrm{mod}\ #1)}

\email{jserrato@usc.edu}
\begin{document}
\title{On the prime spectrum of the $p$-adic integer polynomial ring with a depiction}
\author{Juan Serratos}
\begin{abstract} In $1966$, David Mumford created a drawing of $\proj \zz[X,Y]$ in his book, \textit{Lectures on Curves on an Algebraic Surface}. In following, he created a photo of a so-called \textit{arithmetic surface} $\spec \zz[T]$ for his $1988$ book, \textit{The Red Book of Varieties and Schemes}. The depiction presents the structure of $\spec \zz[T]$ as being interesting and pleasant, and is a well-known picture in algebraic geometry. Taking inspiration from Mumford, we create a drawing similar to $\spec \zz_p[T]$, which has a lot of similarities  with $\spec \zz[T]$.
\end{abstract}
\maketitle
\setcounter{tocdepth}{4}
\setcounter{secnumdepth}{4}
\section{Introduction} 
\subsection{Mumford's picture}The reason for this article is to create a depiction of the prime spectrum of $\zz_p[T]$, in a satisfactory manner, and doesn't claim any possession of originality. There are a few difficulties in doing this because, for a choice of prime $p$, we have a different situation of irreducible polynomials in $\qq_p[T]$ in comparison to other primes, but this is expected. A central example of this surrounds the polynomial  $f(T) = T^2+1$, which Mumford grounded his depiction of $\spec \zz[T]$ with: For a prime $p$ with $p \equiv 1 \Mod{4}$, the polynomial $f(T)$ admits a solution in $\qq_p[T]$, meaning that it is reducible as $f(T)$ is quadratic, so we can't use $f(T)=T^2+1$ in a satisfactory way as Mumford did for $\spec \zz[T]$. 

 \begin{figure*}[htbp] 
	\includegraphics[scale = .8]{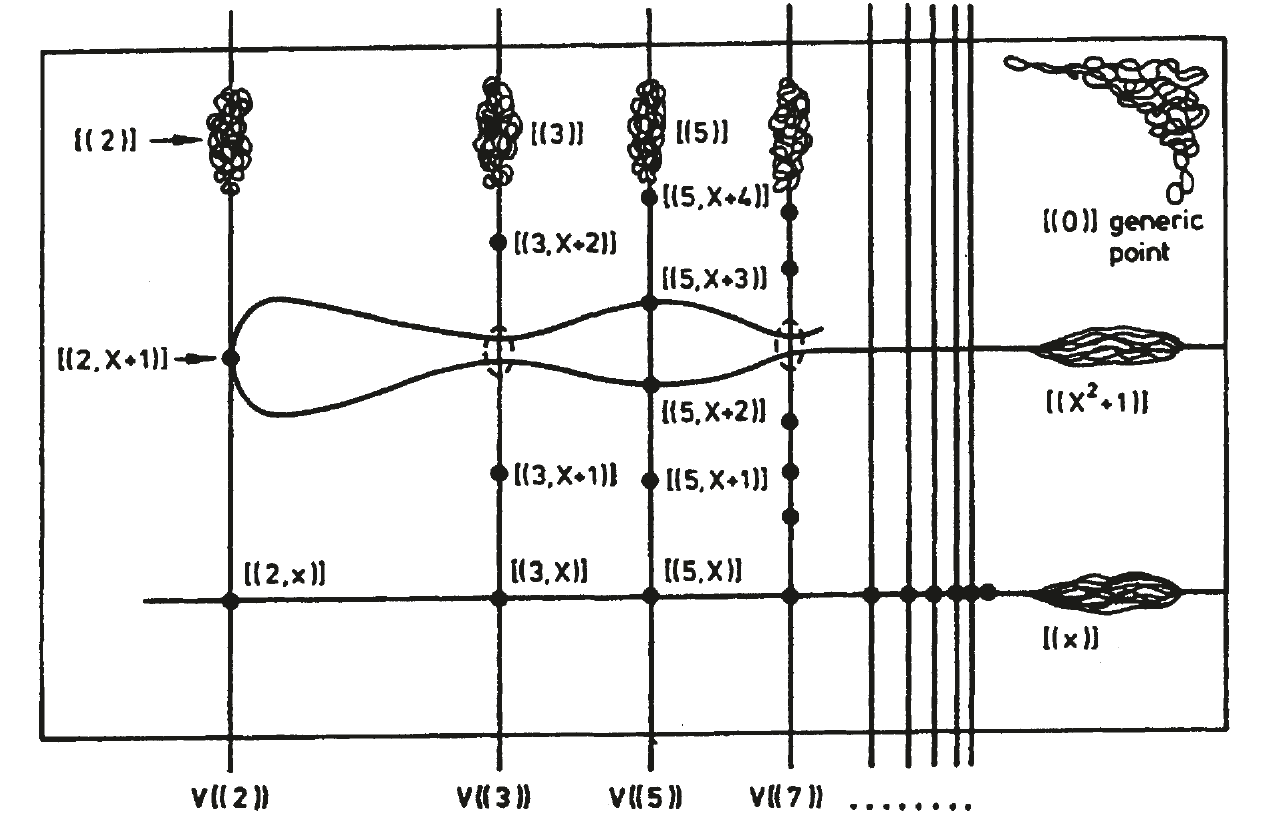}
	\caption{Depiction of $\spec \zz[T]$ ({\protect\cite[II.1, p. 75]{Mum1}}).}
    \label{fig: 1} 
\end{figure*}

Mumford's picture of $\spec \zz[T]$, as shown in Figure \ref{fig: 1}, surrounds quotients of $\zz[T]$. Using a technique in algebraic geometry (which we'll use for $\spec \zz_p[T]$ later), or one could instead doing this by hand, we find out that $\spec \zz[T]$ is in bijection with $\spec \zz/p\zz[T]$ and $\spec \qq[T]$.\footnote{As such we classify $\spec \zz[T]$ as:
\begin{itemize}
	\item [(i)] $(0)$, 
	\item [(ii)] $(p)$ for $p$ prime,
	\item [(iii)] $(f(T))$ with $f(T)$ irreducible in $\zz[T]$, and
	\item [(iii)] $(p, f(T))$ where $p$ is prime and $f(T) \in \ff_p[T]$ is irreducible.
\end{itemize} } In relation to Mumford's picture, we give \textit{greater} importance to those that say more once modded out of $\zz[T]$. Additionally, Mumford has made the point of depicting the ideals of form $(p)$ with some volume as they will contain the points below them on their respective horizontal line, and the same is true of $(T^2+1)$ and $(T)$; recall that when talking about geometry, we invert inclusions. For $(p,f(T))$ where $f(T)$ is $\ff_p$-irreducible, we quotient them with $\zz[T]$ to get: $\zz[T]/(p,f(T)) = (\zz/p\zz [T])/(f(T))= \ff_p[T]/(f(T))$ which results in a field. We cannot depict these all so we just depict the linear polynomials $g(T)$, which are maximal as $\zz[T]/(p,g(T)) = \ff_p[T]/(g(T)) = \ff_p$.
 
 Mumford additionally wanted us to think of $\aff_{\zz}^1$ as a union of the $\aff_{\ff_p}^1$ lines and $\aff_{\qq}^1$ whereby $\aff^1_\qq $ behaves as a horizontal line that interacts with the union of vertical lines of $\aff^1_{\ff_p}$, as outlined in his newest draft, \textit{Algebraic Geometry II (a penultimate draft)} (see {\cite[ 4.1, pp. 119-121]{Mum2}}). The main interaction still deals with taking quotients and inspecting the results. For example, Mumford anchors this philosophy by passing  $f(T) = T^2+1$ throughout the lines of $\aff_{\ff_p}^1$ in the vertical manner and comparing \textit{significances} whether or not drawing big and empty circles ("blips"). He draws a blip, in respect to $f(T)$, on the $\aff_{\ff_7}^1$ line since $\zz [T]/(7,T^2+1) = (\zz/7\zz[T])/(T^2+1) = \ff_{7}[T]/(T^2+1) = \ff_{7^2}$; note that $f(T)$ remains irreducible in $\ff_7[T]$ as it admits not roots. Whereas, he draws no blips on $\aff_{\ff_5}$ or $\aff_{\ff_2}^1$ as $\zz[T]/(2, T^2+1) = \ff_2[T]/(T^2+1) $ and $\zz[T]/(5,T^2+1) = \ff_5[T]/(T^2+1) $ are not even fields: $f(1) = (1)^2+1 \equiv 0 \Mod{2}$ in $\ff_2$ and $f(2) = (2)^2+1 \equiv 0 \Mod {5}$ in $\ff_5$, so neither remain irreducible after processing quotients.

Our main issue is in whether or not to stay close to Mumford's picture of $\spec \zz[T]$ by also anchoring our depiction with $f(T) = T^2+1$ and considering the specific case of $p \not \equiv 1 \Mod {4}$, which is in some ways \textit{unnatural}. Lastly, there is a disappointing note: The intersection of $\spec \qq_p[T]$ and $\spec \qq[T]$ share a common quadratic polynomial, namely, for prime $p$, we have $q(T) = T^2+p \in \spec \qq_p[T] \cap \spec \qq[T]$, yet in terms of our picture of $\spec \zz_p[T]$, it isn't that fascinating. This is because $\zz_p[T]/(p,T^2+p) = (\zz_p/p\zz_p [T])/(T^2+p) = \ff_p[T]/(T^2+p)= \ff_p[T]/(T^2)$, which is not an integral domain; the polynomial $q(T)$ sees a desert of $\aff _{\zz_p}^1$ as its water is not capable of stretching to form an oasis. 
\subsection{Outline} In the following section, we give some background to the \textit{ring of $p$-adic integers} and as well as the \textit{$p$-adic numbers}. It's essential to know two reformulations of Gauss' Lemma and Eisenstein to think about irreducibility of polynomials in $\zz_p[T]$. The reader who has experience with undergraduate algebra and some knowledge about projective limits is perfectly capable, but the reader with only this experience will likely not follow everything once we talk about some algebraic geometry things in $\S$\ref{sec: 3} and should perhaps think of this section as a black box. In addition, we simply state a "stronger" version of Hensel's Lemma and derive Hensel's lemma from there. In $\S$\ref{sec: 3}, we establish, somewhat quickly, what $\spec \zz_p[T]$ is; this comes from a well-known technique from algebraic geometry in using fibres of maps, namely, we look at the fibres of the induced map $\pi \colon \spec \zz_p[T] \to \spec \zz_p$. For the reader who doesn't know about the Zariski topology, we remark here that given a ring $A$, we can endow the space $X = \spec A$, the set of prime ideals of $A$, with a topology. One can think of the cartoons being drawn of the prime spectrums as a picture of a topology instead if they prefer this perspective. We in the end of $\S$\ref{sec: 3} give the picture of $\spec \zz_p[T]$ for $p \not \equiv 1 \Mod {4}$ anchored around $T^2+1$ and present how $T^2+p$ doesn't see a lot of $\spec \zz_p[T]$; the reader themselves can make variations for such a picture of $\spec \zz_p[T]$ for the choice of the $p$ prime or the polynomial they wish to ground the cartoon on. 
\section{\textit{p}-adic numbers}
\label{sec: 2} 
\subsection{Definitions and Basics} One has many ways to start to talk about the $p$-adics. There's the more "analytic" construction that describes the $p$-adic numbers as being formal power series expressions, that is, $\qq_p : = \left \{ \sum_{k=n}^\infty  a_n p^n  \colon a_n \in \ff_p\right \}$, and $\zz_p$ is the subring of $\qq_p$ for which the power series expressions begin at $n=0$. For our purposes, we won't dedicate any more time to this route, but instead we're going to describe $\zz_p$ as a projective limit and its discrete valuation ring description. The \textit{ring of $p$-adic integers} is defined as the projective limit $ \zz_p := \varprojlim _{n} \zz/ p^n \zz $ given by the sequence of of ring maps
\[
\begin{tikzcd}                                               &             &             &   \\
\zz_p \arrow[d] \arrow[rd] \arrow[rrd] \arrow[rrrd, shift left ] &             &             &   \\
\cdots  \arrow[r]                                     & \zz/p^3 \zz  \arrow[r] & \zz/p ^2\zz  \arrow[r] & \zz /p\zz 
\end{tikzcd}
\] where the sequence of maps are given by $a +p^n \mapsto a+p^{n-1} \colon \zz/p^n\zz \to \zz/p^{n-1}\zz$. By construction, since we're taking an inverse limit of ring $\zz/p^n \zz$ for $n \geq 1$, we have that $\zz_p$ will inherit a ring structure from its family of rings which constitute it. We have the following characterization:
\[
\zz_p := \varprojlim_{n} \zz/p^n\zz  = \left \{ (a_n)_{n \geq 1} \colon a_n \in \zz/p^n \zz, \; a_{n+1} \equiv a_n \Mod{p^{n} } \right \}
\]

That is, an element $x \in \zz_p$ is a sequence $x = (x_1, x_2, \ldots)$ whereby $x_{n+1} \equiv x_n \Mod {p^n}$ for all $n \geq 1$. For $x=(x_1, x_2, \ldots )$ and $ y = (y_1, y_2, \ldots)$ in $\zz_p$, we write $x+y = (x_1 + y_1, x_2+y_2, \ldots )$ and $xy = (x_1 y_1, x_2 y_2, \ldots )$, and our multiplicative identity is $1= (1,1,\ldots )$. Lastly, note here that we can embed $\zz \to \zz_p$ by mapping some $x \in \zz$ to $i(x) = (x,x,\ldots )$; that is, $x$ can be represented as $(x \! \mod {p}, x \! \mod p^2, x \! \mod p^3, \ldots) $. For example, for $x = 200 \in \zz$, we embed into $\zz_3$ as $(2,2,11,38,200,200, \ldots)$.
\begin{defi}
The \textit{field  of $p$-adic numbers} $\qq_p$ is the field of fractions of $\zz_p$, i.e. $\qq_p := \zz_p [\frac{1}{p}]$.
\end{defi}
Although we've characterized $\zz_p$ in this algebraic manner, we can once again (quite remarkably) go about this in a different (but equivalent) manner once again. We can describe $\zz_p$ in terms of \textit{valuations}, which is incredibly fruitful and an inescapable tool. We should note here that these reformulations are isomorphic and have an added layer of also a homeomorphism between them. We recall the following definition:
\begin{defi}
	A \textbf{discrete valuation} is a map $v \colon k \to \zz \cup \{ + \infty \} $ on a field $k$ satisfying:
\begin{itemize}
	\item [(i)] $v ( x y) = v (x) + v (y)$, 
	\item [(ii)] $ v (x+y) \geq \min \{ v(x), v (y) \} $, and
	\item [(iii)] $v (x) = \infty \Leftrightarrow x = 0$. 
\end{itemize}
\end{defi}
Before describing the $p$-adic valuation for the field $k = \qq$, we define it as a preliminary step for the integers. If $x \in \zz^\ast = \zz \setminus \{0 \}$, we define $v_p (x)$ to be the unique positive integer satisfying $x = p^{v_p(x)} x^\pp$, where $p$ does not divide $x^\pp$, and $v_p(0) = +\infty $. We extend this to $k= \qq$, by defining $v_p \colon \qq^\ast  \to \zz \cup \{ \infty \}$ by $v_p (\frac{a}{b}) = v_p (a) - v_p (b)$ where $\frac{a}{b} \neq 0$, for which this mapping is indeed a discrete valuation. Furthermore, we define the \textit{$p$-adic absolute value} $|\cdot|_p \colon \qq^\ast \to \rr_{\geq 0}$ by $x\mapsto |x|_p=p^{-v_p(x)}$ and $|0|_p = 0$. This absolute value is in fact \textit{nonarchimedean}, meaning that it satisfies the usual absolute value axioms but has the stronger condition that $|x-y|_p \leq \max \{ |x|_p, |y|_p \} \leq |x|_p + |y|_p$. One then defines a metric $d_p (x,y) = |x-y|_p$ with respect to the $p$-adic absolute value to give us a metric space, and the completion with respect to $d_p$ is the field of $p$-adic numbers. We define the \textit{ring of $p$-adic integers} as $\zz_p := \{ x \in \qq_p \colon |x|_p \leq 1 \} = \{ x \in \qq_p \colon v_p(x) \geq 0 \}$. 
\begin{lemma}\
\begin{itemize}
	\item [(i)] For all $q \in \qq _p^\ast $, we have $q \in \zz_p$ or $q^{-1} \in \zz_p$.
	\item [(ii)] The group of units of $\zz_p$ is $\zz^\times = \{ a \in \zz_p \colon v_p(a) = 0 \} = \{ a \in \zz_p \colon |a|_p = 1 \}$.
\end{itemize}
\label{lem: 2.1}
\end{lemma}

\begin{proof}
(i) Let $q$ be nonzero in $\qq_p$. Then $v_p(1) = v_p(q q^{-1}) = v_p(q)+v_p(q^{-1})=0$, so we must have  that either $v_p(q) \geq$ or $v_p(q^{-1})\geq 0$, and thus $q \in \zz_p$ or $q^{-1} \in \zz_p$. 

(ii) Let $q \in \zz_p$ with $q \neq 0$, so we have $v_p(q) \geq 0$. But $q \in \zz_p^\times$ if and only $q^{-1} \in \zz_p$,  so $v_p(q^{-1}) = -v_p(q) \geq 0$, and if and only if $v_p(q) = 0$. As $v_p(a) =0$, then $|a|_p = 1$ since $|a|_p = \frac{1}{p^0} = 1$, and if $|x|_p =1$, then $1 = p^{v_p(x)} $ so $v_p(x) = 0$. 
\end{proof}
Beyond $\zz_p$ being an integral domain, it enjoys many other structural properties of rings. Consider the set $\mathfrak m_p =\{x\in \qq_p \colon |x|_p  < 1 \} \subset \zz_p$, which is an ideal of $\zz_p$ as $x,y \in \mathfrak m_p$ and $z \in \zz_p$ gives $|zx+y|_p \leq \max \{ |zx|_p, |y|_p \}$ and $|zx|_p = |z|_p |x|_p <1$ so $zx+y \in \mathfrak m_p$. It is straightforward to show $\mathfrak m_p=p \zz_p$. We can show that $\mathfrak m_p$ is in fact maximal and unique, and thus makes $(\zz_p, \mathfrak m _p)$ a local ring. Before we show this, and a few other important qualities of $\zz_p$, note that any $a \in \zz_p$, can be written uniquely as $a = p^n t$ with $t \in \zz_p^\times$ and $n \in \zz_{\geq 0 }$. Since if we let $n = v_p(a)$ then we can write $a = p^n t $ for some $p \nmid t$ by the fundamental theorem of arithmetic. As $p \nmid t$, then $v_p(t) = 0$, so $|t|_p = 1$, and thus by Lemma \ref{lem: 2.1} $t \in \zz_p^\times$, and lastly uniqueness is clear by cancellation. 
\begin{theorem} The integral domain $(\zz_p, \mathfrak m _p)$ is a local principal ideal domain and every nonzero ideal is of the form $I =  p^n\zz_p$.
\label{thm: 2.1}
\end{theorem}
\begin{proof}
	Let $J$ be an ideal of $\zz_p$ with $\mathfrak m_p \subsetneq J \subset \zz_p$. Take $x \in J \setminus \mathfrak m_p $. As $x \in \zz_p$, then $|x|_p \leq 1$ but also since $x \notin \mathfrak m_p$ then we must have $|x|_p= 1$. So $x$ is invertible, and thus $|x^{-1}|_p= \frac{1}{|x|_p} = 1$, meaning $x^{-1} \in \zz_p$. As $J$ is an ideal, and $x \in J$, then $x x^{-1} = 1 \in J$, i.e. $J = \zz_p$. Hence $\mathfrak m_p$ is a maximal ideal of $\zz_p$. To show uniqueness, let $L \subsetneq \zz_p$ be a maximal ideal. Assume that $L$ is distinct from $ \mathfrak m_p$, i.e. there exists $x \in L \setminus \mathfrak m_p$. Then we get a contradiction (following the same argument as above), so there does not exist such an $x$ in their difference. Thus $L \subset \mathfrak m_p$, but as $L$ is maximal then we that either $L = \mathfrak m_p$ or $\mathfrak m_p = \zz_p$, and hence we conclude that $\mathfrak m_p = L$. 
		
	As $(p) \subset \zz_p$ is the generator of $\mathfrak m_{p}$, then every element $x \in \zz_p$ can be written uniquely as $up^n$ where $u \in \zz_p^\times$. Let $a \in I$ be minimal. Then $a = p^n \ell$ with $\ell \in \zz_p ^\times$, but then $p^n$ divides all elements of $I$ and $p^n \in I$, so $I = (p^n) =p^n \zz _p$. 
\end{proof}
\begin{coro}
$ \zz_p$ is a regular local ring, meaning that $\dim \zz_p = 1$ and $\mathfrak m_p$ is generated by a single element.
\end{coro}
\begin{prop}The sequence 
\[
0 \to \zz_p \to \zz_p \to \zz/p^n \zz \to 0 
\]is an exact sequence where $\zz_p \to \zz_p$ is multiplication by $p^n$ and $\zz_p \to \zz/p^n\zz$ where $a = (a_1, a_2, \ldots) \in \zz_p$ is mapped to the $n$-term, i.e. $a \mapsto a_n$.
\end{prop}
\begin{proof}
	Let $\varphi \colon \zz_p \to \zz_p $ be defined by $\varphi (a) = ap^n$. Then, for $a \neq 0$ in $\zz_p$ with $a = (a_1, a_2, \ldots)$, we have that $\varphi (a)=p^na$ implies $(a_1p^n , a_2p^n) = (0,0,\ldots)$, so $a$ must be zero and thus $\ker \varphi = \{ 0 \}$ and $\varphi $ is injective. Denote $\pi_n \colon \zz_p \to \zz/p ^n \zz$ to be the map where $a= (a_1, a_2, \ldots)$ and $\pi_n (a) = a_n$. This is clearly surjective. Hence it remains to show that $\ker \pi_n = \im \varphi$.  We have that $\im \varphi \subset \ker \pi_n$ as $\pi_n \circ \varphi (a) = \pi_n (p^na) =p^n a_n =0$. For $a = (a_1, a_2, \ldots) \in \ker \pi_n$, we have that $\pi_n (a) = a_n= 0$, meaning that $a_n \in p^n \zz_p$. Hence $\ker \pi _n = \im \varphi$ and our sequence is in fact exact. 
\end{proof}
\begin{coro}
	$\zz_p/p^n \zz_p \simeq \zz/p^n\zz$.
\end{coro}
\subsection{Irreducibility in the p-adic integer polynomial ring} A polynomial $p(T) \in \zz_p[T]$ is of the form $p(T) = a_{n}T^n + a_{n-1}T^{n-1}  + \cdots + a_1T +a_0$ where $a_i \in \zz_p$ for $0 \leq i \leq n$. Recall that for an integral domain $A$, a polynomial $p(T) \in A[T]$, which is nonzero and not a unit in $A[T]$, is said to \textit{irreducible} whenever $p(T) = f(T)g(T)$, with $f(T),g(T) \in A[T]$, then either $f(T)$ or $g(T)$ is a unit in $A[T]$. (It is useful to remember the basic fact that, for an integral domain $A$, we have $(A[T])^\times = A^\times$.) We should remember the following fact to further our discussion about $\zz_p[T]$, which was used throughout the entire introduction: 
\begin{lemma}
	Let $A[T]$ be a PID. The polynomial $p(T) \in A[T]$ is irreducible if and only if $A[T]/(p(T))$ is a field. 
\end{lemma}
\begin{proof}
	Assume that $p(T)$ is irreducible. Let $(p(T))$ be contained in the ideal $I = (q(T))$, and we assume $I \neq A[T]$. Then $ p(T) = q(T)g(T)$ for some $g(T) \in A[T]$, so as $p(T)$ is irreducible, then either $q(T)$ or $g(T)$ is invertible. Assume, without loss of generality, $q(T)$ is invertible. Then there exists $f(T) \in A[T]$ such that $q(T)f(T) = 1$. But this implies that $1 \in I$ so $I = A[T]$, which is a contradiction. Thus $(p(T))$ cannot be contained in any other ideal of $A[T]$, so it is maximal and hence $A[T]/(p(T))$ is a field. For the backward direction, let $A[T]/(p(T))$ be a field. Then $(p(T))$ is a maximal ideal. Assume that $p(T)$ is reducible, i.e. $p(T) = f(T) g(T)$ with $\deg f< \deg p$ and $\deg g < \deg p$. So, without loss of generality, we have $(p(T)) \subset (f(T))$. We cannot have that $(f(T)) = A[T]$, as then this would mean that $f(T)$ is invertible and consequentially mean that $p(T)$ is irreducible. Hence $(p(T))$ is not maximal and a contradiction. Therefore $p(T)$ is irreducible.
\end{proof}
From basic algebra, we learn Eisenstein's criterion and Gauss' Lemma. Both of these theorems provide an easier way to verify whether or not a given polynomial $p(T) \in \qq[T]$ is irreducible. There is a necessary condition prevalent in Gauss' Lemma called \textit{primitive}, which means that all the coefficients of our polynomial are relatively prime. This primitive condition is necessary, and relevant to our definition of irreducibility, as there is a difference between irreducibility in $\zz[T]$ and $\qq[T]$. For example, consider $f(T) = 2T^2 + 2 = 2 (T+1)$; this is irreducible over $\qq[T]$ as $2$ is a unit of $\qq$, but this polynomial is reducible over $\zz[T]$ as $2$ is not a unit of $\zz$. In a similar vein, the element $2$ is irreducible in $\zz[T]$, but not in $\qq[T]$ as $2$ is once again just a unit. Now, as $\zz_p [T]$ and $\qq_p[T]$ are the objects most relevant to us, there are in fact analogs of Eisenstein and Gauss' Lemma for these such objects:
\begin{prop}[Eisenstein, {\cite[Proposition 6.3.11]{Gou}}] \label{prop: 2.2}Let $f(T) \in \zz_p[T]$ with $f(T) =a_nT^n + \cdots + a_1 T +a_0 \in \zz_p[T]$ satisfying 
	\begin{itemize}
		\item [(i)] $|a_n |_p = 1$,
		\item [(ii)] $|a_i|_p < 1$ for $0 \leq i < n$, and
		\item [(iii)] $|a_0|_p = \frac{1}{p}$. 
	\end{itemize}Then $f(T)$ is irreducible over $\qq_p[T]$. 
\end{prop}
\begin{ex}
The polynomial $p(T) = T^2-3$ is irreducible over $\qq_3$. 	This polynomial is monic, so $a_n = 1$ and thus $|1|_3  = 1$ as $v_p(1) = 0$ in general, and $a_0 =-3$ is the only other term so we find   $v_3(3) = 1$ meaning that $|3|_3 = \frac{1}{3^1} = \frac{1}{3} < 1$. Therefore $p(x)$ is irreducible in $\qq_3$.  
\end{ex}

\begin{theorem}[Gauss' Lemma\footnote{This is a special case of Gauss' lemma; this in general works over a unique factorization domain and its corresponding field of fractions.}] A non-constant polynomial $f(x) \in \zz_p [T]$ is irreducible in $\zz_p [T]$ if and only it is both irreducible in $\qq_p [T]$ and primitive in $\zz_p[T]$. In particular, if $f$ is irreducible in $\zz_p[T]$, then it is irreducible in $\qq_p[T]$.
\end{theorem}

It is quite remarkable the properties that we've seen come out of $\zz_p$ and $\qq_p$, but what's even more remarkable is that a lot of what we've said generalizes! We'll be vague about this \textit{generalization} as it requires a lot of space to explain, but it surrounds our use of a discrete valuation. What we want to talk about next is \textit{Hensel's lemma}, which loosely speaking says that we can lift a root from $\zz/p\zz$ to a root in $\zz_p$ given some suitable conditions. This generalizes as well! Now, there are "weaker" and "stronger" versions of Hensel's lemma. We will state the stronger without proof (see {\cite[$\S$4, p. 5]{Con}}) and get, as a corollary, the more well-known "weaker" version of Hensel's lemma:

\begin{theorem}[{\cite[Theorem 4.1.]{Con}}]
	Let $f(T) \in \zz_p[T]$ and $a \in \zz_p$ such that 
	\[
	|f(a)|_p < |f^\pp (a)|^2_p.
	\] There is a unique $\alpha \in \zz_p$ such that $f(\alpha) = 0$ in $\zz_p$ and $|\alpha -a|_p < |f^\pp (a)|_p$. Moreover,
	\begin{itemize}
		\item [(1)] $|\alpha-a|_p=|f(a)/f^\pp (a)|_p < |f^\pp (a)|_p$,
		\item [(2)] $|f^\pp (\alpha)|_p = |f^\pp (a)|_p$.
	\end{itemize} 
 \label{Hensel}
\end{theorem}

\begin{coro}[Hensel's lemma] If $f(T) \in \zz_p[T]$ and $a\in \zz_p$ satisfies
\[
f(a) \equiv 0 \Mod {p}, \; f^\pp (a) \not \equiv 0 \Mod{p},
\] then there exists a unique $\alpha \in \zz_p$ such that $f(\alpha) = 0$ in $\zz_p$ and $\alpha \equiv a \Mod{p}$.
\end{coro}
\begin{proof} We have $f^\pp (a) \in \zz_p$ as we're evaluating on a polynomial $f (T) \in \zz_p[T]$ with $a \in \zz_p$, so $|f^\pp (a)|_p \leq 1$ and $|f^\pp (a)|_p^2 \leq 1$. For the special case of $|f^\pp (a)|_p = 1$, we have by Theorem \ref{Hensel} $|f(a)|_p < |f^\pp (a)|_p^2 = 1$. In turn, $|f(a)|_p < 1$ which means that $v_p (f (a)) > 0$, i.e. $p$ divides $f(a)$ at least once. Thus $f (a) \equiv 0 \Mod {p}$. Similarly, as $|f^\pp (a)|_p = 1$ then $v_p (f^\pp (a)) =0$, so $p \nmid  f^\pp (a) $, and hence $f^\pp (a) \not \equiv 0 \Mod {p}$. As $|\alpha - a|_p<|f^\pp (a)|_p =1$, then we get, using the almost exact same argument applied to $f(a)$, that $\alpha \equiv a \Mod{p}$.
\end{proof}
\begin{ex}
Consider the equation $f(T) = T^2-13 $ in $\zz_3$ and take $a= 1 \in \zz_3$. Then $f(1) = 1^2 -13 =-12 \equiv 0 \Mod{3}$ and $f^\pp (T) = 2T$ so $f^\pp (1) = 2 \not \equiv 0 \Mod{3}$. By Hensel's lemma, there exists $\alpha \in \zz_p$ such that $f(\alpha) = \alpha^2 - 13 =0$, i.e. $\alpha^2 = 13$, and $ 	\alpha \equiv 1 \Mod{3}$. This is somewhat strange: we're saying that in $\zz_3$, there exists a number $\alpha \in \zz_3$ so that $13$ is square! This is starkly different from $\zz$ as the equation $f(T)= T^2-13$ is far from having a root, and even its fraction field $\qq$ doesn't possess a root. 
\end{ex}

\section{The prime spectrum of the p-adic integer polynomial ring}
\label{sec: 3}
\subsection{An algebraic geometry technique}It is not too hard to establish what the points of $\spec \zz_p[T]$ are: we will see that $\spec \zz_p[T]$ ends up looking like $\spec \qq_p[T]$ and $\spec \zz/p\zz[T]$. To establish what $\spec \zz_p[T]$ is, we employ a useful tool from algebraic geometry. For the injection $\varphi \colon \zz_p \to \zz_p[T]$, we can get a map on spectra, $\pi \colon  \spec \zz_p [T] \to \spec \zz_p $ where $\mathfrak p \mapsto \varphi ^{-1}(\mathfrak p)$. To inspect the prime ideals of $\zz_p[T]$ amounts to looking at \textit{fibres} of $\pi$. Notice that $\zz_p$  has only two prime ideals, namely $(0)$ and $(p)$ by Theorem \ref{thm: 2.1}, so we only need to look at two fibres of $\pi$. For the map $\pi \colon \spec \zz_p[T] \to \spec \zz_p$ and a point $\mathfrak q \in \spec \zz_p$, \[(\spec \zz_p[T])_\mathfrak q = \spec \zz_p[T] \times _{\spec \zz_p} \spec \kappa (\mathfrak q), \]
where the canonical map $\spec \kappa (\mathfrak q) \to \spec \zz_p$ is given by composition $ \spec \kappa (\mathfrak q ) \to \spec ( \mathscr O_{\spec \zz_p,\mathfrak q } )\to  \spec \zz_p$. The reason this setup is useful is that we can apply the following lemma:
\begin{lemma} Let $f \colon X \to Y$ be a morphism of schemes and let $y \in Y$. Then $X_y = X \times_Y \spec \kappa (y)$ is homeomorphic to $f^{-1}(y)$ with the induced topology. 	
\end{lemma}
To find out $\spec \zz_p[T]$ it amounts to checking $\spec \kappa ((0)) \times _{\zz_p} \spec \zz[T] = \spec (\qq_p \otimes _{\zz_p} \zz[T]) = \spec \qq_p [T]$ and $ \spec \kappa ((p)) \times_{\zz_p} \spec \zz [T] = \spec (\zz/p \zz \otimes_{\zz_p} \zz[T]) = \spec \zz / p \zz[T]$. This follows from the fact that $\mathscr O_{\zz_p, (0)} \simeq (\zz_p)_{(0)} = \qq_p$ and the unique maximal ideal $\mathfrak m  = (0) \zz_p = (0), $ so $\kappa ((0)) =\qq_p /\mathfrak m = \qq_p$, and $\kappa ((p)) = \zz_p/p\zz_p =\zz/p\zz $ follows similarly.  Hence the prime ideals of $ \spec \zz_p [T]$ are in bijection with the prime ideals of $\zz/p \zz[T] \simeq \ff _p [T]$ and $\qq_p [T]$. Thus it suffices to inspect $\spec \zz/p \zz [T]$ and $\spec \qq_p [T]$.\begin{lemma}
$\spec \zz_p[T] = \spec \qq_p[T] \sqcup \coprod\limits_p \spec  \ff_p[T]$; 
that is, $\spec \zz_p[T]$ consists of:
	\begin{itemize}
	\item [(i)] $(0)$, 
	\item [(ii)] $(p)$ for $p$ prime,
	\item [(iii)] $(f(T))$ with $f(T)$-irreducible in $\zz_p[T]$, and
	\item [(iii)] $(p, f(T))$ where $p$ is prime and $f(T) \in \ff_p[T]$ irreducible.
\end{itemize} 
\end{lemma}

Once again, finding some prime/maximal ideals of $\aff_{\zz/p\zz}$ and $\aff_{\qq_p}$ are not hard--in fact, by Proposition \ref{prop: 2.2} we have found that $p(x) = x^2\pm p$ will generate a maximal ideal $(p(x))$ for $\qq_p[T]$ for every $p$. Even easier are the maximal ideals of $\aff_{\zz/p\zz}$ for a fixed $p$: We have $ (p), (p,x), (p,x+1), \ldots, (p,x+(p-1))$ being the maximal ideals landing on $\aff_{\zz/p\zz}$, but these are of course not all.

\subsection{Drawing of \texorpdfstring{$\spec \zz_p[T],p \not \equiv 1 \Mod{4}$}{the prime spectrum} }
\label{subsec: 1}
One should see the figure of $\spec \zz_p[T]$ (see Figure \ref{fig: 2.1}) with $p \not \equiv 1 \Mod{4}$ as separate cases of each $p \not \equiv 1 \Mod{4}$ (see Figure \ref{fig: 3} and Figure \ref{fig: 3.1}) as being stacked on top of each other and the whole global case influencing the choices of $p$. For a choice of $p$ we get the following, each considered detached from the other: 
 
\begin{figure*}[htbp]
\hspace*{-.8cm} 
\begin{tikzpicture}[font=\tiny]

\draw[thick] (7.1,0) ++(9,-.8) node{$\aff_{\qq_p}^1$};
\draw[thick] (2,0) ++(2,-.8) node{$\aff_{\ff_2}^1$};
\draw[thick] (4,0) ++(2,-.8) node{$\aff_{\ff_3}^1$};
\draw[thick] (6,0) ++(2,-.8) node{$\aff_{\ff_7}^1$};
\draw[thick] (8,0) ++(2,-.8) node{$\aff_{\ff_{11}}^1$};

\draw[thick] (3.5,0) -+ (16.5,0);
       \path (4,0) node[dot, label={below:{$(2,T)$}}];
       \path (16,0) node[fuzzy, label={below right:{$(T)$}}];
 		\draw[thick] (4,0) --++(0,5)
		node[fuzzy, label={above left:{$(2)$}}];     
	 
 \draw[thick] (3.5,0) -+ (16.5,0);
       \path (6,0) node[dot, label={below:{$(3,T)$}}];
       \draw[thick] (6,0) --++(0,5)
		node[fuzzy, label={above left:{$(3)$}}];
		\path (6,1.3) node [dot, label={below:{$\;\;\;\;\;\;\;\;\;\;\;\;\;\;\;\;\;\;\;(3,T+1)$}}];
		\path (6,3.5) node [dot, label={below:{$\;\;\;\;\;\;\;\;\;\;\;\;\;\;\;\;\;\;\;(3,T+2)$}}];

 \draw[thick] (3.5,0) -+ (16.5,0);
       \path (8,0) node[dot, label={below:{$(7,T)$}}];
       \path (8,.5) node [dot, label={below:{$\;\;\;\;\;\;\;\;\;\;\;\;\;\;\;\;\;\;\;(7,T+1)$}}];
       \path (8,1.2) node [dot, label={below:{$\;\;\;\;\;\;\;\;\;\;\;\;\;\;\;\;\;\;\; (7,T+2)$}}];
       \path (8,2.5) node [dot, label={below:{$\;\;\;\;\;\;\;\;\;\;\;\;\;\;\;\;\;\;\; (7,T+3)$}}];
       \path (8,3.3) node [dot, label={below:{$\;\;\;\;\;\;\;\;\;\;\;\;\;\;\;\;\;\;\; (7,T+4)$}}];
       \path (8,3.9) node [dot, label={below:{$\;\;\;\;\;\;\;\;\;\;\;\;\;\;\;\;\;\;\; (7,T+5)$}}];
       \path (8,4.45) node [dot, label={below:{$\;\;\;\;\;\;\;\;\;\;\;\;\;\;\;\;\;\;\; (7,T+6)$}}];
       \draw[thick] (8,0) --++(0,5)
		node[fuzzy, label={above left:{$(7)$}}]; 
\draw[thick] (11.5,0) --++(0,5);
\draw[thick] (12,0) --++(0,5);
\draw[thick] (12.3,0) --++(0,5);
\draw[thick] (12.5,0) --++(0,5);
		
 \draw[thick] (3.5,0) -+ (16.5,0);
       \path (10,0) node[dot, label={below:{$(11,T)$}}];
       \draw[thick] (10,0) --++(0,5)
		node[fuzzy, label={above left:{$(11)$}}]; 
		
\path (11.5,0) node[dot];        
\path (12,0) node[dot];
\path (12.3,0) node[dot];
\path (12.5,0) node[dot];

\draw[thick] (16,0) -- ++(0,5) 
        node[fuzzy, inner sep=9pt, label={above right:{$(0)$}}];
  \draw[thick] 
		(4,1.75) node[dot] 
		to[out=270, in=270]
        (6,1.75)
        to[out=0, in=180] (7,1.375) to[out=0, in=180]
        (8, 1.75)
        to[out=0, in=180] (9,1.375) to[out=0, in=180]
        (10, 1.75)
        to[out=0, in=180] (10.75,1.375) to[out=0, in=180]
        (11.5,1.75)
        to[out=0, in=180] (11.75,1.375) to[out=0, in=180]
        (12,1.75) node[big dot]
        to[out=0, in=180] (12.15,1.375) to[out=0, in=180]
        (12.3,1.75) node[big dot]
        to[out=0, in=180] (12.4,1.375) to[out=0, in=180]
        (12.5, 1.75) node[big dot]
        to[out=0, in=180]
        (16,1.75);
   
   \draw[thick]
        (4,1.75) node[dot, pin={190:{$(2, T+1)$}}]
        to[out=90, in=90]
        (6,1.75) node[big dot]
        to[out=0, in=180] (7,2.125) to[out=0, in=180]
        (8, 1.75) node[big dot]
        to[out=0, in=180] (9,2.125) to[out=0, in=180]
        (10, 1.75) node[big dot]
        to[out=0, in=180] (10.75,2.125) to[out=0, in=180]
        (11.5,1.75) node[big dot]
        to[out=0, in=180] (11.75,2.125) to[out=0, in=180]
        (12,1.75) node[big dot]
        to[out=0, in=180] (12.15,2.125) to[out=0, in=180]
        (12.3, 1.75) node[big dot]
        to[out=0, in=180] (12.4,2.125) to[out=0, in=180]
        (12.5, 1.75) node[big dot]
        to[out=0, in=0]
        (13, 1.75) 
        to[out=0, in=0]
        (16,1.75)
        node[fuzzy, pin={190:{$(T^2+1)$}}];
        
\draw[thick]
        (4,2.83) node[dot]
        to[out=0, in=0]
        (6, 2.83) node[dot]
        to[out=0, in=0]
        (8, 2.83) node[dot]
        to[out=0, in=0]
        (10, 2.83) node[dot]
        to[out=0, in=0]
        (11.5,2.83) node[dot]
        to[out=0, in=0]
        (12,2.83) node[dot]
        to[out=0, in=0]
        (12.3, 2.83) node[dot]
        to[out=0, in=0]
        (12.5, 2.83) node[dot]
        to[out=0, in=0]
        (13, 2.83) 
        to[out=0, in=260]
        (16,3.8)
        node[fuzzy, pin={350:{$(T^2+p)$}}];

\end{tikzpicture}
\caption{$\spec \zz_p[T]$ with $p \not \equiv 1 \Mod{4}$}
\label{fig: 2.1}
\end{figure*}

\begin{figure*}[htbp] 
\begin{tikzpicture}[font=\tiny]
\draw[thick] (4,0) ++(9,-.8) node{$\aff_{\qq_2}^1$};
\draw[thick] (4,0) ++(2,-.8) node{$\aff_{\ff_2}^1$};

 \draw[thick] (5,0) -+ (14,0);
       \path (6,0) node[dot, label={below:{$(2,T)$}}];
       \path (13,0) node[fuzzy, label={below right:{$(T)$}}];
        
\draw[thick] (6,0) --++(0,4)
node[fuzzy, label={above left:{$(2)$}}];

   \draw[thick] (13,0) -- ++(0,4) 
        node[fuzzy, inner sep=9pt, label={above right:{$(0)$}}];
  \draw[thick] 
		(6,1.2) node[dot] 
        to[out=60, in=200]
        (13,1.8);
   \draw[thick]
        (6,1.2) node[dot, pin={190:{$(2, T+1)$}}]
        to[out=40, in=200]
        (13,1.8)
        node[fuzzy, pin={220:{$(T^2+1)$}}];
   \draw[thick]
        (6, 2.83) node[dot]
        to[out=0, in=0]
        (13, 2.83)
        node[fuzzy, pin={350:{$(T^2+2)$}}];

\end{tikzpicture} 
\caption{$\spec \zz_2[T]$}
\label{fig: 3.1}
\end{figure*}
\begin{figure*}[htbp]
\begin{tikzpicture}[font=\tiny]
\draw[thick] (4,0) ++(9,-.8) node{$\aff_{\qq_p}^1$};
\draw[thick] (4,0) ++(2,-.8) node{$\aff_{\ff_p}^1$};

 \draw[thick] (5,0) -+ (14,0);
       \path (6,0) node[dot, label={below:{$(p,T)$}}];
       \path (13,0) node[fuzzy, label={below right:{$(T)$}}];
        
\draw[thick] (6,0) --++(0,4)
node[fuzzy, label={above left:{$(p)$}}];

   \draw[thick] (13,0) -- ++(0,4) 
        node[fuzzy, inner sep=9pt, label={above right:{$(0)$}}];
  \draw[thick] 
		(6,1.2) node[dot] 
        to[out=60, in=200]
        (13,1.8);
   \draw[thick]
        (6,1.2) node[big dot, pin={190:{$(p, T+1)$}}]
        to[out=2, in=250]
        (13,1.8)
        node[fuzzy, pin={350:{$(T^2+1)$}}];
   \draw[thick]
        (6, 2.83) node[dot]
        to[out=0, in=0]
        (13, 2.83)
        node[fuzzy, pin={350:{$(T^2+p)$}}];

\end{tikzpicture}
\caption{$\spec \zz_p[T]$ for $p>2$ and $p \not \equiv 1 \Mod{4}$}
\label{fig: 3}
\end{figure*}
\begin{figure*}[htbp]
\hspace*{-.8cm}  
\begin{tikzpicture}[font=\tiny]

\draw[thick] (7.1,0) ++(9,-.8) node{$\aff_{\qq_p}^1$};
\draw[thick] (2,0) ++(2,-.8) node{$\aff_{\ff_2}^1$};
\draw[thick] (4,0) ++(2,-.8) node{$\aff_{\ff_3}^1$};
\draw[thick] (6,0) ++(2,-.8) node{$\aff_{\ff_7}^1$};
\draw[thick] (8,0) ++(2,-.8) node{$\aff_{\ff_{11}}^1$};

\draw[thick] (3.5,0) -+ (16.5,0);
       \path (4,0) node[dot, label={below:{$(2,T)$}}];
       \path (16,0) node[fuzzy, label={below right:{$(T)$}}];
 		\draw[thick] (4,0) --++(0,5)
		node[fuzzy, label={above left:{$(2)$}}];     
	 
 \draw[thick] (3.5,0) -+ (16.5,0);
       \path (6,0) node[dot, label={below:{$(3,T)$}}];
       \draw[thick] (6,0) --++(0,5)
		node[fuzzy, label={above left:{$(3)$}}];
		\path (6,1.3) node [dot, label={below:{$\;\;\;\;\;\;\;\;\;\;\;\;\;\;\;\;\;\;\;(3,T+1)$}}];
		\path (6,3.5) node [dot, label={below:{$\;\;\;\;\;\;\;\;\;\;\;\;\;\;\;\;\;\;\;(3,T+2)$}}];

 \draw[thick] (3.5,0) -+ (16.5,0);
       \path (8,0) node[dot, label={below:{$(7,T)$}}];
       \path (8,.5) node [dot, label={below:{$\;\;\;\;\;\;\;\;\;\;\;\;\;\;\;\;\;\;\;(7,T+1)$}}];
       \path (8,1.2) node [dot, label={below:{$\;\;\;\;\;\;\;\;\;\;\;\;\;\;\;\;\;\;\; (7,T+2)$}}];
       \path (8,2.5) node [dot, label={below:{$\;\;\;\;\;\;\;\;\;\;\;\;\;\;\;\;\;\;\; (7,T+3)$}}];
       \path (8,3.3) node [dot, label={below:{$\;\;\;\;\;\;\;\;\;\;\;\;\;\;\;\;\;\;\; (7,T+4)$}}];
       \path (8,3.9) node [dot, label={below:{$\;\;\;\;\;\;\;\;\;\;\;\;\;\;\;\;\;\;\; (7,T+5)$}}];
       \path (8,4.45) node [dot, label={below:{$\;\;\;\;\;\;\;\;\;\;\;\;\;\;\;\;\;\;\; (7,T+6)$}}];
       \draw[thick] (8,0) --++(0,5)
		node[fuzzy, label={above left:{$(7)$}}]; 
\draw[thick] (11.5,0) --++(0,5);
\draw[thick] (12,0) --++(0,5);
\draw[thick] (12.3,0) --++(0,5);
\draw[thick] (12.5,0) --++(0,5);
		
 \draw[thick] (3.5,0) -+ (16.5,0);
       \path (10,0) node[dot, label={below:{$(11,T)$}}];
       \draw[thick] (10,0) --++(0,5)
		node[fuzzy, label={above left:{$(11)$}}]; 
		
\path (11.5,0) node[dot];        
\path (12,0) node[dot];
\path (12.3,0) node[dot];
\path (12.5,0) node[dot];

\draw[thick] (16,0) -- ++(0,5) 
        node[fuzzy, inner sep=9pt, label={above right:{$(0)$}}];
  \draw[thick] 
		(4,1.75) node[dot] 
		to[out=270, in=270]
        (6,1.75)
        to[out=0, in=180] (7,1.375) to[out=0, in=180]
        (8, 1.75)
        to[out=0, in=180] (9,1.375) to[out=0, in=180]
        (10, 1.75)
        to[out=0, in=180] (10.75,1.375) to[out=0, in=180]
        (11.5,1.75)
        to[out=0, in=180] (11.75,1.375) to[out=0, in=180]
        (12,1.75) node[big dot]
        to[out=0, in=180] (12.15,1.375) to[out=0, in=180]
        (12.3,1.75) node[big dot]
        to[out=0, in=180] (12.4,1.375) to[out=0, in=180]
        (12.5, 1.75) node[big dot]
        to[out=0, in=180]
        (16,1.75);
   
   \draw[thick]
        (4,1.75) node[dot, pin={190:{$(2, T+1)$}}]
        to[out=90, in=90]
        (6,1.75) node[big dot]
        to[out=0, in=180] (7,2.125) to[out=0, in=180]
        (8, 1.75) node[big dot]
        to[out=0, in=180] (9,2.125) to[out=0, in=180]
        (10, 1.75) node[big dot]
        to[out=0, in=180] (10.75,2.125) to[out=0, in=180]
        (11.5,1.75) node[big dot]
        to[out=0, in=180] (11.75,2.125) to[out=0, in=180]
        (12,1.75) node[big dot]
        to[out=0, in=180] (12.15,2.125) to[out=0, in=180]
        (12.3, 1.75) node[big dot]
        to[out=0, in=180] (12.4,2.125) to[out=0, in=180]
        (12.5, 1.75) node[big dot]
        to[out=0, in=0]
        (13, 1.75) 
        to[out=0, in=0]
        (16,1.75)
        node[fuzzy, pin={190:{$(T^2+1)$}}];
        
\draw[thick]
        (4,2.83) node[dot]
        to[out=0, in=0]
        (6, 2.83) node[dot]
        to[out=0, in=0]
        (8, 2.83) node[dot]
        to[out=0, in=0]
        (10, 2.83) node[dot]
        to[out=0, in=0]
        (11.5,2.83) node[dot]
        to[out=0, in=0]
        (12,2.83) node[dot]
        to[out=0, in=0]
        (12.3, 2.83) node[dot]
        to[out=0, in=0]
        (12.5, 2.83) node[dot]
        to[out=0, in=0]
        (13, 2.83) 
        to[out=0, in=260]
        (16,3.8)
        node[fuzzy, pin={350:{$(T^2+p)$}}];

\draw [-stealth](10,-1) -- (10,-2);

\draw[thick] (7.1,-8) ++(9,-.8) node{$\aff_{\qq}^1$};
\draw[thick] (2,-8) ++(2,-.8) node{$\aff_{\ff_2}^1$};
\draw[thick] (4,-8) ++(2,-.8) node{$\aff_{\ff_3}^1$};
\draw[thick] (6,-8) ++(2,-.8) node{$\aff_{\ff_5}^1$};
\draw[thick] (8,-8) ++(2,-.8) node{$\aff_{\ff_{7}}^1$};

\draw[thick] (3.5,0) -+ (16.5,0);
       \path (4,-8) node[dot, label={below:{$(2,T)$}}];
       \path (16,-8) node[fuzzy, label={below right:{$(T)$}}];
 		\draw[thick] (4,-8) --++(0,5)
		node[fuzzy, label={above left:{$(2)$}}];     
\draw[thick] (3.5,-8) -+ (16.5,-8);
       \path (4,-8) node[dot, label={below:{$(2,T)$}}];
       \path (16,-8) node[fuzzy, label={below right:{$(T)$}}];
 		\draw[thick] (4,-8) --++(0,5)
		node[fuzzy, label={above left:{$(2)$}}];

 \draw[thick] (3.5,-8) -+ (16.5,-8);
       \path (6,-8) node[dot, label={below:{$(3,T)$}}];
       \draw[thick] (6,-8) --++(0,5)
		node[fuzzy, label={above left:{$(3)$}}];
		\path (6,.69-8) node [dot, label={below:{$\;\;\;\;\;\;\;\;\;\;\;\;\;\;\;\;\;\;\;(3,T+1)$}}];
		\path (6,3.5-7.9) node [dot, label={below:{$\;\;\;\;\;\;\;\;\;\;\;\;\;\;\;\;\;\;\;(3,T+2)$}}];

 \draw[thick] (3.5,-8) -+ (16.5,-8);
       \path (10,-8) node[dot, label={below:{$(7,T)$}}];
       \path (10,.4-8) node [dot, label={below:{$\;\;\;\;\;\;\;\;\;\;\;\;\;\;\;\;\;\;\;(7,T+1)$}}];
       \path (10,1.2-8.3) node [dot, label={below:{$\;\;\;\;\;\;\;\;\;\;\;\;\;\;\;\;\;\;\; (7,T+2)$}}];
       \path (10,2.5-8.]) node [dot,  pin={20:{$(7, T+3)$}}];
       \path (10,3.5-7.9) node [dot, label={below:{$\;\;\;\;\;\;\;\;\;\;\;\;\;\;\;\;\;\;\; (7,T+4)$}}];
       \path (10,4-7.93) node [dot, label={below:{$\;\;\;\;\;\;\;\;\;\;\;\;\;\;\;\;\;\;\; (7,T+5)$}}];
       \path (10,4.5-8) node [dot, label={below:{$\;\;\;\;\;\;\;\;\;\;\;\;\;\;\;\;\;\;\; (7,T+6)$}}];
       \draw[thick] (8,-8) --++(0,5)
		node[fuzzy, label={above left:{$(5)$}}]; 

\draw[thick] (3.5,-8) -+ (16.5,-8);
       \path (8,-8) node[dot, label={below:{$(5,T)$}}];
       \path (8,.-7.6) node [dot, label={below:{$\;\;\;\;\;\;\;\;\;\;\;\;\;\;\;\;\;\;\;(5,T+1)$}}];
       \path (8,.5899-7.5999) node [dot, label={below:{$\;\;\;\;\;\;\;\;\;\;\;\;\;\;\;\;\;\;\; (5,T+2)$}}];
       \path (8,2.92-8.42) node [dot, pin={20:{$(5, T+3)$}}];
       \path (8,3.6-8) node [dot, label={below:{$\;\;\;\;\;\;\;\;\;\;\;\;\;\;\;\;\;\;\; (5,T+4)$}}];
       \path (8,4.3-8) node [dot, label={below:{$\;\;\;\;\;\;\;\;\;\;\;\;\;\;\;\;\;\;\; (5,T+5)$}}];

\draw[thick] (11.5,-8) --++(0,5);
\draw[thick] (12,0-8) --++(0,5);
\draw[thick] (12.3,0-8) --++(0,5);
\draw[thick] (12.5,0-8) --++(0,5);

       \draw[thick] (10,-8) --++(0,5)
		node[fuzzy, label={above left:{$(7)$}}];
\path (11.5,0-8) node[dot];        
\path (12,0-8) node[dot];
\path (12.3,0-8) node[dot];
\path (12.5,0-8) node[dot];

 \draw[thick]
        (4,1.75-8) node[dot]
        to[out=90, in=90]
        (6,1.75-8) node[big dot]
        to[out=0, in=180] (8,15-20.5) to[out=0, in=180]
        (10,1.75-8) node[big dot]
        to[out=0, in=180] (10.75,15-20.5) to[out=0, in=180]
        (11.5,1.75-8) 
        to[out=90, in=90]
        (12,1.75-8) 
        to[out=90, in=90]
        (12.3, 1.75-8)
        to[out=90, in=90]
        (12.5, 1.75-8)
        to[out=0, in=0]
        (13, 1.75-8) 
        to[out=0, in=0]
        (16,1.75-8);
 \draw[thick]
        (4,1.75-8) node[dot, pin={190:{$(2, T+1)$}}]
        to[out=270, in=270]
        (6,1.75-8) node[big dot]
        to[out=0, in=180] (8,8-15) to[out=0, in=180]
        (10, 1.75-8) node[big dot]
        to[out=0, in=180] (10.75,8-15) to[out=0, in=180]
        (11.5,1.75-8) node[big dot]
        to[out=270, in=270]
        (12,1.75-8)  
        to[out=270, in=270]
        (12.3, 1.75-8) 
        to[out=270, in=270]
        (12.5, 1.75-8)node[big dot]
        to[out=0, in=0]
        (13, 1.75-8) 
        to[out=0, in=0]
        (16,1.75-8)
        node[fuzzy, pin={190:{$(T^2+1)$}}];

\draw[thick] (16,0-8) -- ++(0,5) 
        node[fuzzy, inner sep=9pt, label={above right:{$(0)$}}];
\end{tikzpicture}
\caption{If one squints their eyes, we can see an embedding of $\spec \zz_p[T]$ to  $\spec \zz[T]$ for $p \not \equiv 1 \Mod {4}$. }	
\end{figure*}

\bibliographystyle{plain}

\end{document}